\documentclass[10pt]{article}
\usepackage{fancyhdr}
\usepackage{extramarks}
\usepackage{amsmath}
\usepackage{amsthm}
\usepackage[utf8]{inputenc}   % arXiv/pdflatex-friendly
\usepackage[T1]{fontenc}
\usepackage{newunicodechar}   % map a few Unicode run-ins to TeX
 \usepackage{dblfloatfix} % improves placement of figure* in two-column docs
\usepackage{float}

\usepackage{booktabs}
\usepackage{tabularx}

\newunicodechar{α}{$\alpha$}
\newunicodechar{β}{$\beta$}
\newunicodechar{γ}{$\gamma$}
\newunicodechar{κ}{$\kappa$}
\newunicodechar{δ}{$\delta$}

\newunicodechar{–}{--}       % en dash
\newunicodechar{—}{---}      % em dash
\newunicodechar{’}{'}        % apostrophe
\newunicodechar{±}{$\pm$}

\newunicodechar{­}{}
\usepackage{amsfonts}
\usepackage{siunitx}
\usepackage{tikz}
\usepackage[plain]{algorithm}
\usepackage{algpseudocode}
\usepackage{multirow}
\usepackage{booktabs}
\usepackage{graphicx}
\usepackage{subfigure}
\usepackage[margin=1in]{geometry}
\usepackage{booktabs}
\usepackage{makecell}
\usepackage{array} 
\usepackage[colorlinks,linkcolor=black,anchorcolor=black,citecolor=black,urlcolor=blue]{hyperref}
\usepackage{hyphenat}
\usepackage{amsmath,bm}
\usepackage{booktabs}
\usepackage{mathtools}
\usepackage{amssymb}
\usepackage{tikz-cd}
\usepackage{caption}
\usepackage{capt-of}
\usepackage{mciteplus}
\usepackage{cite}
\usepackage{mathrsfs}
\usepackage[title,titletoc,toc]{appendix}
\usepackage{xr}
\usepackage{parskip}
\usepackage{soul}
\usepackage{textcomp}
\usepackage[colaction]{multicol}
\usepackage[switch]{lineno}
\usepackage{lipsum}
\usepackage{etoolbox}
\usepackage{longtable}
\usepackage{array}
\usepackage{tablefootnote}
\usepackage{ragged2e}
\usepackage{soul}
\newcolumntype{C}[1]{>{\centering\arraybackslash}p{#1}}
\usetikzlibrary{automata,positioning}
\usepackage{float}

\usetikzlibrary{automata,positioning}

\newtheorem{remark}{Remark}[section]

\newtheorem{lemma}{Lemma}[section]

\newtheorem{proposition}{Proposition}[section]
\newtheorem{corollary}{Corollary}[section]

\def\Gr{{\mbox{\bf Gr}}}

\def\arg{{\mbox{arg}\,}}

\usepackage{amsthm}  %  
\theoremstyle{definition}
\newtheorem{definition}{Definition}[section]

\usepackage{xr}
\begin{document}
	\title{Data-Adaptive Grassmann Manifold Representations for Spatial Transcriptomics Alignment}
    %\title{Precision Drug Repurposing for Alzheimer's Disease Enabled via Persistent Sheaf Laplacian Analysis of Gene Co-Expression Networks }
    
	\author{Xiang Xiang Wang$^{1}$, Sean Cottrell$^{1,2}$, and Guo-Wei Wei$^{1,3,4}$\footnote{
			Corresponding author.		Email: guowei.wei@uga.edu} \\% Author name
		\\
		$^1$ Department of Mathematics, \\
		Michigan State University, East Lansing, MI 48824, USA.\\
        $^2$ Department of Computational Mathematics, Science, and Engineering, \\
        Michigan State University, East Lansing, MI 48824, USA. \\	
        $^3$ Department of Mathematics, \\ University of Georgia, Athens, GA  30602, USA.\\
		$^4$Department of Biochemistry and Molecular Biology, \\ University of Georgia, Athens, GA  30602, USA.\\
	}
	\date{June 10, 2026} % Date for the report
	
	\maketitle
	
 % SIAM Article Template

\begin{abstract}
Spatial transcriptomics measures gene expression together with spatial coordinates, but many existing analysis methods represent each spot primarily by a single feature vector. We propose GrassST, a  subspace method for spatial transcriptomics analysis and cross-slice alignment. For each spatial spot, GrassST constructs a neighborhood from tissue coordinates, fits a low-dimensional subspace to the embedded expression profiles in that neighborhood, and represents the spot by the resulting  subspace. These representations are points on a Grassmann manifold and can be compared using distances between subspaces. GrassST selects the neighborhood size and subspace rank from the spectral energy of the data, allowing these parameters to vary across datasets rather than being fixed globally. Experiments on four spatial transcriptomics datasets show that GrassST achieves competitive clustering and cross-slice integration performance under a unified evaluation pipeline.

\end{abstract}


\noindent\textbf{MSCcodes:} 14M15, 53B20,  65D18
	
	%{\setcounter{tocdepth}{4} \tableofcontents}
	%\setcounter{page}{1}

% 

 \section{Introduction}

% Spatial transcriptomics measures gene expression while retaining spatial information from tissue sections, allowing transcriptional profiles to be analyzed together with their tissue locations \cite{Stahl2016Science,Vickovic2019NatMethods}. Such data can be expressed as pairs $(x_i,s_i)$, where $x_i\in\mathbb{R}^d$ denotes a high-dimensional expression vector and $s_i\in\mathbb{R}^2$ denotes the corresponding spatial coordinate: each location carries a high-dimensional feature vector, and nearby locations are expected to share local structure induced by tissue organization. 
% When multiple tissue sections are analyzed jointly, a central task is spatial transcriptomics alignment: constructing a common representation in which related tissue regions across sections can be compared while accounting for both transcriptional similarity and spatial organization \cite{Zeira2022NatMethods,zhou2023integrating}.
%  This setting raises the representation problem studied in this paper: how can stable local features be constructed from noisy high-dimensional observations so that neighborhood-level variation is preserved and aligned across tissue sections?

Single-cell RNA sequencing has become a standard approach for profiling cellular heterogeneity, identifying cell populations, and studying transcriptional state changes in complex tissues \cite{trapnell2014dynamics,haghverdi2016diffusion,wolf2018scanpy,choi2019dissecting, qiu2022mapping}. Its rapid development has also motivated a broad range of computational methods for clustering, dimensionality reduction, generative modeling, and geometric analysis of cell-state structure \cite{lopez2018deep,tian2019clustering,cottrell2024k,huynh2024topological}. However, dissociation-based single-cell measurements generally remove cells
from their native tissue context, so the spatial locations of cells and their
local neighborhood relationships are not directly observed
\cite{Stahl2016Science,Vickovic2019NatMethods}.
 This limitation motivates spatially resolved transcriptomic technologies, which aim to measure gene expression while retaining single cells' location geometric organization.

 Spatial transcriptomics measures gene expression while retaining spatial information from microenvironment, allowing transcriptional profiles to be analyzed together with their tissue locations \cite{Stahl2016Science,Vickovic2019NatMethods}. Such data can be expressed as pairs $(x_i^{(b)},s_i^{(b)})$, where $x_i^{(b)}\in\mathbb{R}^d$ is the gene expression vector of spot $i$ in tissue section $b$, and $s_i^{(b)}\in\mathbb{R}^2$ is the corresponding spatial coordinate.
 Local neighborhoods may contain related transcriptional patterns due to tissue organization. When multiple tissue sections are analyzed jointly, a central task is spatial transcriptomics alignment: constructing a common representation in which related tissue regions across sections can be compared while accounting for both transcriptional similarity and spatial organization \cite{Zeira2022NatMethods,zhou2023integrating}. This setting motivates the need for representations that can support reliable alignment across tissue sections.

% This setting raises a basic representation problem. Given noisy high-dimensional observations on one or more spatial domains, how should one construct stable local geometric features that preserve neighborhood-level variation and allow comparison across spatial locations and tissue sections?

Most existing computational methods for spatial transcriptomics represent each spot or cell by a pointwise feature vector, obtained either from gene expression measurements or from learned latent embeddings. Spatial information is then incorporated through graph smoothing, neighborhood aggregation, attention mechanisms, contrastive learning, or probabilistic alignment models \cite{Hu2021NatMethodsSpaGCN,Dong2022NatCommunSTAGATE,Li2023NatCommunPRECAST,Long2023NatCommunGraphST,zhou2023integrating,lin2024contrastive}. 
These methods have been effective for spatial domain identification and multi-slice integration. However, from a geometric point of view, representing each spot separately may be insufficient for characterizing the local organization of spatial neighborhoods. In particular, local neighborhoods can contain coherent directions of variation that are not captured by comparing single expression vectors alone.

This observation suggests a subspace-based representation for each spatial spot. Around such a spot, a local neighborhood is constructed from nearby spatial locations. If the embedded expression profiles within this neighborhood vary mainly along a few directions, then the center spot can be represented by a low-dimensional subspace fitted to its neighborhood. Comparing these spot-level subspaces then provides a way to compare local tissue structure across sections. The fitted subspace records the main directions of local expression change around the center spot, while reducing the influence of directions with small variation. Distances between these fitted subspaces can then be used to compare the local expression structure around spots from different sections.

 The appropriate mathematical space for these objects is the Grassmann manifold $\Gr(r,p)$, whose points are $p$-dimensional linear subspaces of $\mathbb{R}^r$ \cite{Bendokat2024,Edelman1998SIAM,Absil2008Book}. Grassmannian representations have been widely used in numerical linear algebra, signal processing, computer vision, and manifold learning \cite{Edelman1998SIAM,Absil2008Book,Dong2014TSP,WangTam2026JMIV}. They have also been applied to multi-omics integration and cancer patient stratification \cite{Ding2019Bioinformatics,Alfatemi2022BMC}. Representative spatial transcriptomics methods, including SEDR \cite{SEDR2023NatCommun}, STAGATE \cite{Dong2022NatCommunSTAGATE}, and GraphST \cite{Long2023NatCommunGraphST}, commonly rely on deep spatial embeddings or graph-based representations rather than  Grassmannian subspaces. More broadly, applied mathematical approaches have been introduced to model structure in high-dimensional biological and scientific data, including supervised optimal transport for transport-based matching with prior information and supervised Gromov--Wasserstein optimal transport for structured matching with metric-preserving constraints \cite{CangNieZhao2022SIAM,CangWuZhao2025SIAMMDS}. Recent work on geometric representations for single-cell and spatial data further supports the use of structured mathematical representations for transcriptomic data \cite{Maehara2025NatCommunDdHodge,Cottrell2025AdvSciMCIST,wang2026multiscale}.

In this work, we propose GrassST, a Grassmann subspace method for spatial transcriptomics alignment. The method first maps all expression profiles into a shared feature space. For each spot, GrassST constructs a spatial neighborhood using tissue coordinates and fits a low-dimensional subspace to the embedded expression profiles in that neighborhood. The fitted subspace is assigned to the center spot as its local representation. Thus, for spot $i$ in section $b$, GrassST defines
\[
\Phi_{k,p}: s_i^{(b)} \mapsto \mathcal{U}_{b,i} \in \Gr(r,p),
\]
where $\mathcal{U}_{b,i}$ is the $p$-dimensional subspace fitted to the $k$-neighborhood of spot $i$ in section $b$. Pairwise distances between these subspaces are then used to compare local expression patterns across spots and sections. These distances are computed from projection matrices and therefore, do not depend on the particular orthonormal basis chosen for each subspace.

Parameters $k$ and $p$ play a central role in the construction. The neighborhood scale $k$ determines the spatial resolution of the local covariance estimate. Small neighborhoods preserve fine spatial structure but may yield unstable covariance estimates, whereas large neighborhoods improve sampling stability but may mix distinct tissue regions. The rank $p$ determines the dimension of the retained subspace. A small rank may discard meaningful directions of variation, whereas a large rank may include directions associated with small eigenvalues and reduce the stability of the resulting Grassmann manifold representation. Thus, the quality of $\Phi_{k,p}$ depends on a coupled choice of spatial scale and subspace dimension.

To address this issue, we develop a data-adaptive procedure for selecting both $k$ and $p$ from the local spectral structure of the data. This avoids using fixed neighborhood scales and subspace ranks across datasets. The full selection rule is described in Section~\ref{subsec:adaptive}.
The selection procedure is motivated by the stability of the fitted subspaces. Neighborhood scales and subspace ranks that better capture the dominant local variation tend to produce more stable Grassmann manifold  representations. The stability analysis is provided in Section~\ref{subsec:perturbation}.

The main contributions of this work are as follows. First, we formulate a Grassmann manifold subspace representation for spatially indexed high-dimensional transcriptomic data, where each spatial spot is represented by a subspace fitted to its local neighborhood. Second, we derive the GrassST construction, which maps spatial spots to points on $\Gr(r,p)$ through subspace estimation and Grassmann manifold distance. Third, we introduce a data-adaptive selection procedure for the neighborhood scale $k$ and subspace rank $p$ based on cumulative spectral energy and stability considerations. Fourth, we validate the resulting representation on multiple spatial transcriptomics datasets, including DLPFC, HER2-positive breast cancer, Barista-seq, and MERFISH data, and examine its performance in clustering, cross-slice alignment, and robustness analyses.

The rest of the paper is organized as follows. Section 2 reviews the mathematical foundations used in the proposed framework, including Grassmann manifolds, projector representations, and distances between subspaces. Section 3 introduces the GrassST construction and the adaptive selection procedure for the neighborhood scale and subspace rank. Section 4 reports numerical experiments on spatial transcriptomics datasets. Section 5 discusses the methodological implications and limitations. Section 6 concludes the paper.

\section{Mathematical Foundations}
\label{sec:math_foundations}

This section introduces the notation and geometric background used in the
GrassST construction, including the Grassmann manifold, projector
representations, and projection-based distances between subspaces. The main
notation is summarized in Table~\ref{tab:notation}.

\subsection{Grassmann Manifolds and Projector Representation}
\label{subsec:grassmann}

For integers $1 \leq p \leq r$, the Grassmann manifold $\Gr(r,p)$ is the set of all $p$-dimensional linear subspaces of $\mathbb{R}^r$:
\[
\Gr(r,p)
=
\{\mathcal{U} \subset \mathbb{R}^r : \dim(\mathcal{U})=p\}.
\]
An element $\mathcal{U}\in\Gr(r,p)$ can be represented by an orthonormal basis matrix
\[
U \in \mathbb{R}^{r\times p},
\qquad
U^\top U=I_p,
\]
whose columns span $\mathcal{U}$. This representation is not unique. If $Q\in \mathrm{O}(p)$ is an orthogonal matrix, then $U$ and $UQ$ span the same subspace.

\begin{table}[htbp]
\centering
\small
\begin{tabular}{c|p{0.8\textwidth}}
\toprule
\textbf{Symbol} & \textbf{Description} \\
\midrule
$B$ & Number of tissue sections \\
$n_b$ & Number of spatial spots in section $b$ \\
$N=\sum_{b=1}^{B}n_b$ & Total number of spatial spots across all sections \\
$X^{(b)}\in\mathbb{R}^{n_b\times d}$ & Gene expression matrix for section $b$ \\
$x_i^{(b)}\in\mathbb{R}^d$ & Expression vector of spot $i$ in section $b$ \\
$s_i^{(b)}\in\mathbb{R}^2$ & Spatial coordinate of spot $i$ in section $b$ \\
$z_i^{(b)}\in\mathbb{R}^r$ & Embedded feature vector of spot $i$ in section $b$ \\
$r$ & Shared ambient feature dimension after embedding \\
$k$ & Neighborhood patch size \\
$\mathcal{N}_k(b,i)$ & $k$-nearest spatial neighborhood of spot $i$ in section $b$ \\
$p$ & Subspace rank \\
$\Gr(r,p)$ & Grassmann manifold of $p$-dimensional subspaces of $\mathbb{R}^r$ \\
$U_{b,i}\in\mathbb{R}^{r\times p}$ & Orthonormal basis matrix for the subspace associated with spot $(b,i)$ \\
$\mathcal{U}_{b,i}$ & Subspace associated with spot $(b,i)$ \\
$P_{b,i}=U_{b,i}U_{b,i}^{\top}$ & Orthogonal projector onto $\mathcal{U}_{b,i}$ \\
$\Sigma_{b,i}^{(k)}$ & Centered local covariance matrix for neighborhood $\mathcal{N}_k(b,i)$ \\
$E_p^{(b,i,k)}$ & Cumulative spectral energy at rank $p$ for neighborhood $\mathcal{N}_k(b,i)$ \\
$\bar{E}_p(k)$ & Dataset-level average cumulative spectral energy at rank $p$ and patch size $k$ \\
\bottomrule
\end{tabular}
\caption{Notation used in this paper.}
\label{tab:notation}
\end{table}

A basis-independent representation is given by the orthogonal projector
\[
P_U = UU^\top .
\]
The matrix $P_U$ satisfies
\[
P_U^\top=P_U,
\qquad
P_U^2=P_U,
\qquad
\operatorname{rank}(P_U)=p .
\]
Conversely, every symmetric idempotent matrix of rank $p$ is the orthogonal projector onto a unique $p$-dimensional subspace. Thus, the Grassmann manifold can equivalently be represented as
\[
\Gr(r,p)
=
\{P\in\mathbb{R}^{r\times r}: P^\top=P,\; P^2=P,\; \operatorname{rank}(P)=p\}.
\]
This projector representation is particularly convenient for numerical computation because it removes the nonuniqueness of the basis matrix and allows subspaces to be compared through matrices in an ambient Euclidean space \cite{Edelman1998SIAM,Absil2008Book,Bendokat2024}.

\subsection{Distances Between Subspaces}
\label{subsec:distance}

Let $\mathcal{U},\mathcal{V}\in\Gr(r,p)$ have orthonormal basis matrices
$U,V\in\mathbb{R}^{r\times p}$. The principal angles
\[
\theta_1,\ldots,\theta_p \in [0,\pi/2]
\]
between $\mathcal{U}$ and $\mathcal{V}$ are defined by
\[
\cos\theta_\ell=\sigma_\ell(U^\top V),
\qquad
\ell=1,\ldots,p,
\]
where $\sigma_\ell(\cdot)$ denotes the $\ell$th singular value, ordered
nonincreasingly.

Distances on the Grassmann manifold can be expressed in terms of these
principal angles. One standard choice is the geodesic distance
\[
d_{\mathrm{geo}}(\mathcal{U},\mathcal{V})
=
\left(
\sum_{\ell=1}^{p}\theta_\ell^2
\right)^{1/2}.
\]
Another common choice is the chordal distance
\[
d_{\mathrm{ch}}(\mathcal{U},\mathcal{V})
=
\left(
\sum_{\ell=1}^{p}\sin^2\theta_\ell
\right)^{1/2}.
\]
Both distances depend only on the subspaces and not on the particular orthonormal bases used to represent them. In the numerical experiments, GrassST uses the chordal distance $d_{\mathrm{ch}}$ to compare local
subspaces. Other standard Grassmannian distances are discussed in Refs.
\cite{HammLee2008,YeLim2016,Bendokat2024}.

\subsection{Projection Residuals and Principal Subspaces}
\label{subsec:principal_subspace}

We recall two standard facts used in the subspace estimation step of
GrassST. The projection residual decomposition follows from elementary
properties of orthogonal projectors, while the Ky Fan maximum principle gives
the variational characterization of leading eigenspaces. These results are
standard in matrix analysis and numerical linear algebra
\cite{GolubVanLoan2013,StewartSun1990}.

\begin{lemma}[Orthogonal projection residual decomposition]
\label{lem:projection_residual}
Let $U\in\mathbb{R}^{r\times p}$ satisfy $U^\top U=I_p$, and let
$P_U=UU^\top$. Then $P_U$ is the orthogonal projector onto
$\operatorname{span}(U)$. For any $y\in\mathbb{R}^r$,
\[
\|y-P_Uy\|_2^2
=
\|y\|_2^2-\|U^\top y\|_2^2 .
\]
\end{lemma}

\begin{proof}
Since $U^\top U=I_p$, we have
\[
P_U^\top=P_U,
\qquad
P_U^2=P_U.
\]
Thus $P_U$ is an orthogonal projector. The two vectors $P_Uy$ and
$(I-P_U)y$ are orthogonal because
\[
(P_Uy)^\top (I-P_U)y
=
y^\top P_U(I-P_U)y
=
y^\top(P_U-P_U^2)y
=
0.
\]
Since
\[
y=P_Uy+(I-P_U)y,
\]
the Pythagorean identity gives
\[
\|y\|_2^2
=
\|P_Uy\|_2^2+\|(I-P_U)y\|_2^2.
\]
Therefore,
\[
\|y-P_Uy\|_2^2
=
\|y\|_2^2-\|P_Uy\|_2^2.
\]
Finally,
\[
\|P_Uy\|_2^2
=
y^\top P_U^\top P_Uy
=
y^\top P_Uy
=
y^\top UU^\top y
=
\|U^\top y\|_2^2.
\]
Combining the two identities proves the claim.
\end{proof}

\begin{lemma}[Ky Fan maximum principle]
\label{lem:ky_fan}
Let $\Sigma\in\mathbb{R}^{r\times r}$ be symmetric, with eigenvalues
\[
\lambda_1(\Sigma)\geq \lambda_2(\Sigma)\geq \cdots \geq \lambda_r(\Sigma).
\]
Then
\[
\max_{\substack{U\in\mathbb{R}^{r\times p}\\ U^\top U=I_p}}
\operatorname{tr}(U^\top \Sigma U)
=
\sum_{\ell=1}^{p}\lambda_\ell(\Sigma).
\]
The maximum is attained by choosing the columns of $U$ to be orthonormal
eigenvectors associated with the $p$ largest eigenvalues of $\Sigma$.
\end{lemma}

\section{Subspace Representations on the Grassmann manifold with Adaptive Neighborhood Scale and Subspace Rank}
\label{sec:framework}

Each spatial location in a tissue section carries a high-dimensional gene expression profile. Rather than representing such a location by a single vector in a Euclidean feature space, we associate to it a $p$-dimensional linear subspace of a shared embedding, estimated from the local spatial neighborhood. The collection of these subspaces forms a point cloud on the Grassmann manifold $\Gr(r,p)$, a compact Riemannian symmetric space that admits intrinsic comparison across spatial locations and across tissue sections.

The cross-slice structure is implicit in this construction. A neighborhood graph defined jointly over all sections couples the subspace estimates across sections before any manifold embedding is performed, so that the resulting point cloud on $\Gr(r,p)$ inherits a geometric consistency across sections without requiring explicit registration of the coordinate systems.

\subsection{The GrassST Construction}
\label{subsec:overview}

The GrassST construction is illustrated in Fig.~\ref{fig:frame}. We consider a dataset of $B$ tissue sections. For section $b$, let
\[
X^{(b)} \in \mathbb{R}^{n_b \times d}
\]
denote the gene expression matrix, where $n_b$ is the number of spatial locations and $d$ is the number of measured genes. The spatial coordinates are
\[
S^{(b)} = \{s_1^{(b)}, \ldots, s_{n_b}^{(b)}\},
\qquad s_i^{(b)} \in \mathbb{R}^2 .
\]
The full dataset is
\[
\mathcal{D} = \{(X^{(b)}, S^{(b)})\}_{b=1}^B,
\qquad
N = \sum_{b=1}^B n_b .
\]
GrassST constructs a data-dependent map from indexed spatial locations to the Grassmann manifold,
\[
\Phi_{k,p}: \{(b,i): 1\leq b\leq B,\; 1\leq i\leq n_b\}
\longrightarrow \Gr(r,p),
\qquad
(b,i) \mapsto \mathcal{U}_{b,i}.
\]
Equivalently, writing the indexed location as $s_i^{(b)}$, this assignment can be denoted by
\[
s_i^{(b)} \mapsto \mathcal{U}_{b,i}.
\]
Here $k$ is the neighborhood scale, $p$ is the subspace rank, and $\mathcal{U}_{b,i}$ is the local principal subspace estimated from the embedded expression profiles in the spatial neighborhood of $s_i^{(b)}$. The goal of the construction is to compare spatial locations through local covariance geometry rather than through individual expression vectors alone.

\medskip

\noindent\textbf{Stage 1: Shared feature embedding.}
Expression profiles from different sections are first placed in a common ambient space. Concatenating all sections gives
\[
X^{\mathrm{all}}
=
\begin{bmatrix}
X^{(1)};
\cdots;
X^{(B)}
\end{bmatrix}
\in \mathbb{R}^{N \times d}.
\]
After preprocessing, a linear map
\[
T:\mathbb{R}^d \to \mathbb{R}^r
\]
is obtained from the leading $r$ principal components of $X^{\mathrm{all}}$, yielding
\[
Z^{(b)} = T(X^{(b)}),
\qquad
Z^{\mathrm{all}} \in \mathbb{R}^{N \times r}.
\]
Because $T$ is estimated jointly across all sections, every embedded profile $z_i \in \mathbb{R}^r$ lies in a single coordinate system. This provides a consistent ambient space for the subspace construction.

\medskip

\noindent\textbf{Stage 2: Spatial neighborhood graph.}
For each indexed spot $(b,i)$, define its spatial neighborhood by
\[
\mathcal{N}_k(b,i)
=
\operatorname{kNN}\bigl(s_i^{(b)}; S\bigr),
\]
where \(s_i^{(b)}\in\mathbb{R}^2\) is the spatial coordinate of spot \(i\) in
section \(b\), and \(S\) denotes the coordinate set used for neighborhood
construction. In the default setting,
\[
S=S^{(b)}
=
\begin{bmatrix}
(s_1^{(b)})^\top;
\cdots;
(s_{n_b}^{(b)})^\top
\end{bmatrix}
\in\mathbb{R}^{n_b\times 2},
\]
so neighborhoods are constructed within each tissue section:
\[
\mathcal{N}_k(b,i)
\subset
\{(b,j):1\leq j\leq n_b\}.
\]

If sections are registered or otherwise represented in comparable coordinates,
\(S\) may instead be taken as the pooled coordinate matrix
\[
S^{\mathrm{all}}
=
\begin{bmatrix}
S^{(1)};
\cdots;
S^{(B)}
\end{bmatrix}
\in \mathbb{R}^{N\times 2},
\qquad
N=\sum_{b=1}^{B}n_b .
\]
In this cross-slice setting,
\[
\mathcal{N}_k(b,i)
\subset
\{(b',j):1\leq b'\leq B,\;1\leq j\leq n_{b'}\},
\]
and the neighborhood of spot $(b,i)$ may contain spots from multiple sections.
The scale \(k\) controls the spatial resolution of the representation and is
selected adaptively in Section~\ref{subsec:adaptive}.

\medskip

\noindent\textbf{Stage 3: Subspace estimation.}
Equipped with the neighborhood graph $\{\mathcal{N}_k(b,i)\}$, we estimate a subspace for each spatial spot in each tissue section. Let
\[
Z_{b,i}^{(k)}
=
\{z_{b',j}: (b',j)\in\mathcal{N}_k(b,i)\}
\subset \mathbb{R}^r
\]
denote the embedded expression profiles in the $k$-neighborhood of spot $i$ in section $b$. A local PCA is fitted to this neighborhood, and the first $p$ loading vectors are used as an orthonormal basis for the subspace. Equivalently, if
\[
\bar z_{b,i}^{(k)}
=
\frac{1}{|\mathcal{N}_k(b,i)|}
\sum_{(b',j)\in\mathcal{N}_k(b,i)} z_{b',j}
\]
is the mean of the neighborhood features and
\[
y_{b',j}^{(b,i,k)}
=
z_{b',j}-\bar z_{b,i}^{(k)}
\]
are the centered features, then the local basis $U_{b,i}\in\mathbb{R}^{r\times p}$ solves
\begin{equation}
\label{eq:local_subspace}
U_{b,i}
=
\arg\min_{\substack{U\in\mathbb{R}^{r\times p}\\ U^\top U=I_p}}
\sum_{(b',j)\in\mathcal{N}_k(b,i)}
\left\|
y_{b',j}^{(b,i,k)} - UU^\top y_{b',j}^{(b,i,k)}
\right\|_2^2 .
\end{equation}
This is the standard variational characterization of PCA; in the implementation, the centering is performed internally by the PCA routine.

Equivalently, $U_{b,i}$ is given by the leading $p$ eigenvectors of the centered local covariance matrix
\[
\Sigma_{b,i}^{(k)}
=
\frac{1}{|\mathcal{N}_k(b,i)|}
\sum_{(b',j)\in\mathcal{N}_k(b,i)}
y_{b',j}^{(b,i,k)}
\bigl(y_{b',j}^{(b,i,k)}\bigr)^\top .
\]
The corresponding subspace is
\[
\mathcal{U}_{b,i}
=
\operatorname{span}(U_{b,i})
\in \Gr(r,p).
\]
Repeating this construction for all spatial spots in all sections gives
\[
\{\mathcal{U}_{b,i}:1\leq b\leq B,\;1\leq i\leq n_b\}
\subset \Gr(r,p).
\]
The rank $p$ controls the dimension of each subspace and is selected together with $k$ in Section~\ref{subsec:adaptive}.

\medskip

\noindent\textbf{Stage 4: Intrinsic comparison via Grassmann manifold geometry.}
To compare two local neighborhoods, GrassST uses the projection-Frobenius distance
\[
d_{\mathrm{ch}}(\mathcal{U}_i,\mathcal{U}_j)
=
\|U_i U_i^\top - U_j U_j^\top\|_F .
\]
This distance depends only on the orthogonal projectors onto the two subspaces and is therefore independent of the choice of orthonormal basis. Equivalently, if $\theta_1,\ldots,\theta_p$ are the principal angles between $\mathcal{U}_i$ and $\mathcal{U}_j$, then
\[
d_{\mathrm{ch}}^2(\mathcal{U}_i,\mathcal{U}_j)
=
2\sum_{\ell=1}^{p}\sin^2\theta_\ell .
\]
The pairwise distance matrix
\[
D_{k,p}=(d_{ij}),
\qquad
d_{ij}=d_{\mathrm{ch}}(\mathcal{U}_i,\mathcal{U}_j),
\]
quantifies subspace discrepancy across the dataset.

\medskip

\noindent\textbf{Stage 5: Grassmann manifold representation for cross-slice alignment.}
The map $\Phi_{k,p}$ embeds all spatial locations from multiple tissue sections into a common Grassmannian representation space. In this space, two locations are close in $\Gr(r,p)$ when their local transcriptional neighborhoods share similar dominant subspace structure, regardless of their original slice identity. GrassST therefore aligns spots across sections by comparing their subspaces through the Grassmannian distance structure $D_{k,p}$, rather than directly matching raw gene-expression vectors or spatial coordinates.

The resulting Grassmannian representation supports cross-slice alignment by mixing slice-specific batches while preserving spatial domain organization. It can also be used to construct affinity graphs, perform spectral or graph-based clustering, visualize local geometric organization, and compare locations within or across tissue sections. Thus, the primary object produced by GrassST is the map $\Phi_{k,p}$ together with the induced distance structure $D_{k,p}$, which provides the geometric basis for alignment and downstream analysis.

\begin{figure}[h]
\centering
\includegraphics[width=\textwidth]{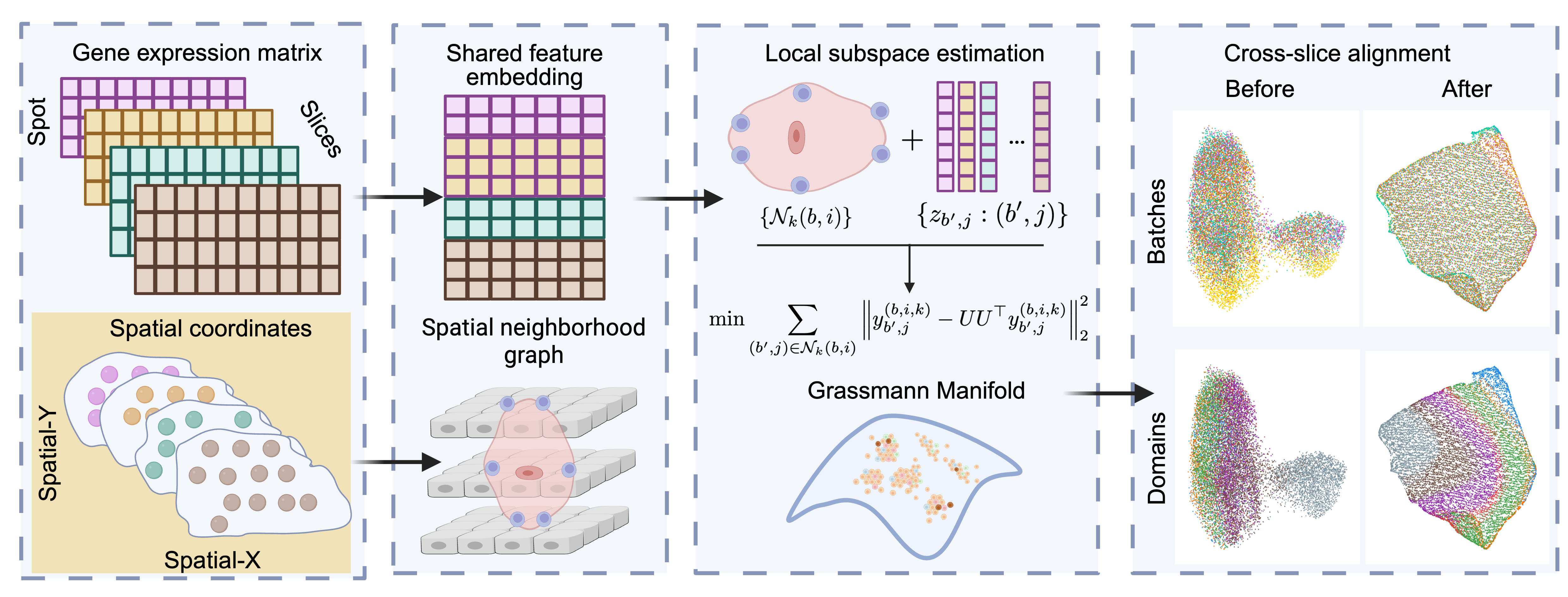}
\caption{
Overview of the GrassST framework for spatial transcriptomics alignment. For each tissue section, GrassST takes as input a gene expression matrix and the corresponding spatial coordinates. The expression matrices from all sections are first mapped into a shared feature embedding, so that spots from different slices are represented in the same ambient space. The spatial coordinates are used to construct a neighborhood graph. For an indexed spot $(b,i)$, where $b$ denotes the tissue section and $i$ denotes the spot within that section, $\mathcal{N}_k(b,i)$ denotes its $k$-nearest spatial neighborhood. The embedded expression vectors in this neighborhood are denoted by $\{z_{b',j}:(b',j)\in\mathcal{N}_k(b,i)\}$, where $(b',j)$ indexes a neighboring spot that may come from the same or another section. After centering these neighborhood features, $y_{b',j}^{(b,i,k)}$ denotes the centered feature vector associated with neighbor $(b',j)$ relative to the neighborhood of $(b,i)$. GrassST fits a local $p$-dimensional subspace by solving the minimization problem
$
U_{b,i}
=
\arg\min_{\substack{U\in\mathbb{R}^{r\times p}\\ U^\top U=I_p}}
\sum_{(b',j)\in\mathcal{N}_k(b,i)}
\left\|
y_{b',j}^{(b,i,k)} - UU^\top y_{b',j}^{(b,i,k)}
\right\|_2^2.$
The resulting column space $\mathcal{U}_{b,i}=\operatorname{span}(U_{b,i})$ is a point on the Grassmann manifold $\Gr(r,p)$. Pairwise Grassmannian distances between these subspaces are then used to compare neighborhood-level structure across sections and support cross-slice alignment. The right panels illustrate the embedding before and after alignment, colored either by batch or by spatial domain.
}
\label{fig:frame}
\end{figure}

\subsection{Subspace Estimation}
\label{subsec:local_grassmann}

The local Grassmann manifold representation is obtained from the dominant eigenspace of a centered local covariance matrix. For an indexed spot $(b,i)$ and neighborhood scale $k$, let
\[
\mathcal{N}_k(b,i)
\]
denote the spatial neighborhood of spot $i$ in section $b$, and let
\[
\bar z_{b,i}^{(k)}
=
\frac{1}{|\mathcal{N}_k(b,i)|}
\sum_{(b',j)\in \mathcal{N}_k(b,i)} z_{b',j}
\]
be the local mean. Define the centered neighborhood features
\[
y_{b',j}^{(b,i,k)}
=
z_{b',j}-\bar z_{b,i}^{(k)},
\qquad
(b',j)\in\mathcal{N}_k(b,i),
\]
and the empirical local covariance matrix
\begin{equation}
\label{eq:local_cov}
\Sigma_{b,i}^{(k)}
=
\frac{1}{|\mathcal{N}_k(b,i)|}
\sum_{(b',j)\in\mathcal{N}_k(b,i)}
y_{b',j}^{(b,i,k)} \bigl(y_{b',j}^{(b,i,k)}\bigr)^\top
\in \mathbb{R}^{r\times r}.
\end{equation}
The centering step ensures that the subspace captures variation within the neighborhood rather than the location of the neighborhood mean in the ambient feature space.

For a prescribed rank $p$, the subspace is defined as the solution of
\begin{equation}
\label{eq:local_subspace_estimation}
U_{b,i}
=
\arg\min_{\substack{U\in\mathbb{R}^{r\times p}\\ U^\top U=I_p}}
\sum_{(b',j)\in\mathcal{N}_k(b,i)}
\left\|
y_{b',j}^{(b,i,k)} - UU^\top y_{b',j}^{(b,i,k)}
\right\|_2^2 .
\end{equation}
This objective selects the $p$-dimensional subspace that minimizes the total projection residual of the centered neighborhood features.

\begin{proposition}
\label{prop:pca_solution}
Let
\[
\lambda_1(\Sigma_{b,i}^{(k)})\geq \lambda_2(\Sigma_{b,i}^{(k)})\geq \cdots
\geq \lambda_r(\Sigma_{b,i}^{(k)})\geq 0
\]
be the eigenvalues of $\Sigma_{b,i}^{(k)}$, with corresponding orthonormal
eigenvectors
\[
u_1,\ldots,u_r .
\]
A solution to~\eqref{eq:local_subspace_estimation} is
\[
U_{b,i} = [u_1\mid \cdots \mid u_p],
\]
and the associated subspace is
\[
\mathcal{U}_{b,i}
=
\operatorname{span}(u_1,\ldots,u_p)
\in \Gr(r,p).
\]
\end{proposition}

\begin{proof}
Let $U\in\mathbb{R}^{r\times p}$ satisfy $U^\top U=I_p$, and set
$P_U=UU^\top$. By Lemma~\ref{lem:projection_residual}, for each
$(b',j)\in\mathcal{N}_k(b,i)$,
\[
\left\|
y_{b',j}^{(b,i,k)} - UU^\top y_{b',j}^{(b,i,k)}
\right\|_2^2
=
\|y_{b',j}^{(b,i,k)}\|_2^2
-
\|U^\top y_{b',j}^{(b,i,k)}\|_2^2 .
\]
Summing over $(b',j)\in\mathcal{N}_k(b,i)$ gives
\[
\sum_{(b',j)\in\mathcal{N}_k(b,i)}
\left\|
y_{b',j}^{(b,i,k)} - UU^\top y_{b',j}^{(b,i,k)}
\right\|_2^2
=
\sum_{(b',j)\in\mathcal{N}_k(b,i)}
\|y_{b',j}^{(b,i,k)}\|_2^2
-
\sum_{(b',j)\in\mathcal{N}_k(b,i)}
\|U^\top y_{b',j}^{(b,i,k)}\|_2^2 .
\]
The first term is independent of $U$. We now rewrite the second term in terms
of the local covariance matrix. For each $(b',j)$,
\[
\|U^\top y_{b',j}^{(b,i,k)}\|_2^2
=
\operatorname{tr}
\left(
(U^\top y_{b',j}^{(b,i,k)})(U^\top y_{b',j}^{(b,i,k)})^\top
\right).
\]
Using
\[
(U^\top y_{b',j}^{(b,i,k)})(U^\top y_{b',j}^{(b,i,k)})^\top
=
U^\top y_{b',j}^{(b,i,k)}
\bigl(y_{b',j}^{(b,i,k)}\bigr)^\top U,
\]
we obtain
\[
\|U^\top y_{b',j}^{(b,i,k)}\|_2^2
=
\operatorname{tr}
\left(
U^\top y_{b',j}^{(b,i,k)}
\bigl(y_{b',j}^{(b,i,k)}\bigr)^\top U
\right).
\]
Therefore,
\[
\begin{aligned}
\sum_{(b',j)\in\mathcal{N}_k(b,i)}
\|U^\top y_{b',j}^{(b,i,k)}\|_2^2
&=
\sum_{(b',j)\in\mathcal{N}_k(b,i)}
\operatorname{tr}
\left(
U^\top y_{b',j}^{(b,i,k)}
\bigl(y_{b',j}^{(b,i,k)}\bigr)^\top U
\right)  \\
&=
\operatorname{tr}
\left(
U^\top
\left[
\sum_{(b',j)\in\mathcal{N}_k(b,i)}
y_{b',j}^{(b,i,k)}
\bigl(y_{b',j}^{(b,i,k)}\bigr)^\top
\right]
U
\right).
\end{aligned}
\]
By the definition of $\Sigma_{b,i}^{(k)}$,
\[
\sum_{(b',j)\in\mathcal{N}_k(b,i)}
y_{b',j}^{(b,i,k)}
\bigl(y_{b',j}^{(b,i,k)}\bigr)^\top
=
|\mathcal{N}_k(b,i)|\Sigma_{b,i}^{(k)}.
\]
Hence
\[
\sum_{(b',j)\in\mathcal{N}_k(b,i)}
\|U^\top y_{b',j}^{(b,i,k)}\|_2^2
=
|\mathcal{N}_k(b,i)|
\operatorname{tr}
\bigl(U^\top \Sigma_{b,i}^{(k)}U\bigr).
\]

Thus the objective in~\eqref{eq:local_subspace_estimation} can be written as
\[
\sum_{(b',j)\in\mathcal{N}_k(b,i)}
\left\|
y_{b',j}^{(b,i,k)} - UU^\top y_{b',j}^{(b,i,k)}
\right\|_2^2
=
C_{b,i}
-
|\mathcal{N}_k(b,i)|
\operatorname{tr}
\bigl(U^\top \Sigma_{b,i}^{(k)}U\bigr),
\]
where
\[
C_{b,i}=
\sum_{(b',j)\in\mathcal{N}_k(b,i)}
\|y_{b',j}^{(b,i,k)}\|_2^2
\]
does not depend on $U$. Therefore, minimizing the projection residual over
all $U$ with orthonormal columns is equivalent to maximizing
\[
\operatorname{tr}
\bigl(U^\top \Sigma_{b,i}^{(k)}U\bigr)
\]
over the same constraint set.

By Lemma~\ref{lem:ky_fan}, this maximum is attained when the columns of $U$
are orthonormal eigenvectors corresponding to the $p$ largest eigenvalues of
$\Sigma_{b,i}^{(k)}$. Therefore, one solution is
\[
U_{b,i}=[u_1\mid\cdots\mid u_p].
\]
The associated subspace is the column space of $U_{b,i}$, namely
\[
\mathcal{U}_{b,i}
=
\operatorname{span}(u_1,\ldots,u_p)
\in\Gr(r,p).
\]
\end{proof}

Since only the column space of $U_{b,i}$ is used, the local representation is
invariant under a change of orthonormal basis. If $R\in\mathrm{O}(p)$, then
$U_{b,i}$ and $U_{b,i}R$ represent the same point on $\Gr(r,p)$. A basis-independent
representative is the orthogonal projector
\begin{equation}
\label{eq:projector}
P_{b,i}
=
U_{b,i}U_{b,i}^\top
\in\mathbb{R}^{r\times r}.
\end{equation}
This projector satisfies
\[
P_{b,i}^2=P_{b,i},
\qquad
P_{b,i}^\top=P_{b,i},
\qquad
\operatorname{rank}(P_{b,i})=p,
\]
and is unchanged under replacement $U_{b,i}\mapsto U_{b,i}R$.

For two subspaces $\mathcal{U}_{b,i}$ and $\mathcal{U}_{c,j}$, let
\(\theta_1,\ldots,\theta_p\) denote their principal angles. In numerical
experiments, GrassST compares subspaces using the chordal distance
\[
d_{\mathrm{ch}}(\mathcal{U}_{b,i},\mathcal{U}_{c,j})
=
\left(
\sum_{\ell=1}^{p}\sin^2\theta_\ell
\right)^{1/2}.
\]
Equivalently, in terms of the corresponding projectors,
\[
d_{\mathrm{ch}}(\mathcal{U}_{b,i},\mathcal{U}_{c,j})
=
\frac{1}{\sqrt{2}}
\|P_{b,i}-P_{c,j}\|_F .
\]
Thus, the distance is independent of the particular orthonormal bases used to
represent the two subspaces. The projector representation will also be used in
Section~\ref{subsec:perturbation} to quantify the sensitivity of
$\mathcal{U}_{b,i}$ to perturbations of $\Sigma_{b,i}^{(k)}$.

\subsection{Stability of the Subspace Representation}
\label{subsec:perturbation}

The local Grassmannian representation assigned to spot $(b,i)$ is obtained
from the leading eigenspace of the centered local covariance matrix
$\Sigma_{b,i}^{(k)}$. Perturbations in the embedded expression profiles
therefore induce perturbations in $\Sigma_{b,i}^{(k)}$, which may change the
corresponding subspace $\mathcal{U}_{b,i}\in\Gr(r,p)$. The stability of this
construction is controlled by the separation between the retained and
discarded eigenvalues.

Let
\[
\lambda_1(\Sigma)\geq \lambda_2(\Sigma)\geq \cdots \geq \lambda_r(\Sigma)
\]
denote the eigenvalues of a symmetric positive semidefinite matrix $\Sigma$.
For a target rank $p$, define the eigengap
\[
\delta_p(\Sigma)
=
\lambda_p(\Sigma)-\lambda_{p+1}(\Sigma).
\]
The following result is a direct consequence of the Davis--Kahan
$\sin\Theta$ theorem.

\begin{proposition}[Eigenspace perturbation]
\label{prop:dk}
Let $\Sigma,\widetilde{\Sigma}\in\mathbb{R}^{r\times r}$ be symmetric
positive semidefinite matrices, and let
$\mathcal{U},\widetilde{\mathcal{U}}\in\Gr(r,p)$ denote their leading
$p$-dimensional invariant subspaces. Suppose that
\[
\delta_p(\Sigma)>0
\]
and that the perturbation is small enough so that the leading
$p$-dimensional invariant subspace of $\widetilde{\Sigma}$ remains separated
from its complement. Then
\[
d_{\mathrm{ch}}(\mathcal{U},\widetilde{\mathcal{U}})
\leq
C
\frac{\|\Sigma-\widetilde{\Sigma}\|_F}{\delta_p(\Sigma)}
\]
for a universal constant $C>0$.
\end{proposition}

\begin{proof}
Let $U$ and $\widetilde U$ be orthonormal bases for
$\mathcal{U}$ and $\widetilde{\mathcal{U}}$, respectively. Let
\[
\Theta(\mathcal{U},\widetilde{\mathcal{U}})
=
\operatorname{diag}(\theta_1,\ldots,\theta_p)
\]
denote the principal-angle matrix between the two subspaces. The
Davis--Kahan $\sin\Theta$ theorem gives
\[
\|\sin\Theta(\mathcal{U},\widetilde{\mathcal{U}})\|_F
\leq
C
\frac{\|\Sigma-\widetilde{\Sigma}\|_F}{\delta_p(\Sigma)}
\]
for a universal constant $C>0$, under the stated eigenvalue separation
condition. By the definition of the chordal distance,
\[
d_{\mathrm{ch}}(\mathcal{U},\widetilde{\mathcal{U}})
=
\left(
\sum_{\ell=1}^{p}\sin^2\theta_\ell
\right)^{1/2}
=
\|\sin\Theta(\mathcal{U},\widetilde{\mathcal{U}})\|_F .
\]
Combining the two identities gives the stated bound
\cite{DavisKahan1970,StewartSun1990,YuWangSamworth2015}.
\end{proof}

Applying Proposition~\ref{prop:dk} to the local covariance matrices used in
GrassST gives the following corollary.

\begin{corollary}[Subspace stability]
\label{cor:local_stability}
Let $\Sigma_{b,i}^{(k)}$ and $\widetilde{\Sigma}_{b,i}^{(k)}$ be two
centered local covariance matrices associated with spot $(b,i)$ and
neighborhood scale $k$. Let
$
\mathcal{U}_{b,i},\widetilde{\mathcal{U}}_{b,i}\in\Gr(r,p)
$
be their leading $p$-dimensional subspaces. If
$
\delta_p(\Sigma_{b,i}^{(k)})>0
$
and the corresponding perturbed eigenspace remains separated, then
\[
d_{\mathrm{ch}}(\mathcal{U}_{b,i},\widetilde{\mathcal{U}}_{b,i})
\leq
C
\frac{
\|\Sigma_{b,i}^{(k)}-\widetilde{\Sigma}_{b,i}^{(k)}\|_F
}{
\delta_p(\Sigma_{b,i}^{(k)})
}.
\]
\end{corollary}

Corollary~\ref{cor:local_stability} shows that the perturbation of a local
Grassmannian representation is controlled by two quantities: the magnitude of
the perturbation in the local covariance matrix and the eigengap separating
the retained and discarded directions. A larger eigengap yields a more stable
subspace, whereas a small eigengap makes the leading subspace more
sensitive to perturbations. Since
\[
d_{\mathrm{ch}}(\mathcal{U},\widetilde{\mathcal{U}})
\leq
\sqrt{p},
\]
the representation remains well-defined even when the eigengap is small,
although the perturbation bound becomes less informative in that regime.

Fig.~\ref{fig:subspace_stability} provides a numerical illustration of the
subspace perturbation bound using DLPFC slice 151673
\cite{Maynard2021DLPFC}, including a comparison between the subspace
representation and the PCA embedding of the center spot.

\begin{figure}[hpt!]
\centering
 \makebox[\textwidth][c]{
\includegraphics[width=0.85\textwidth]{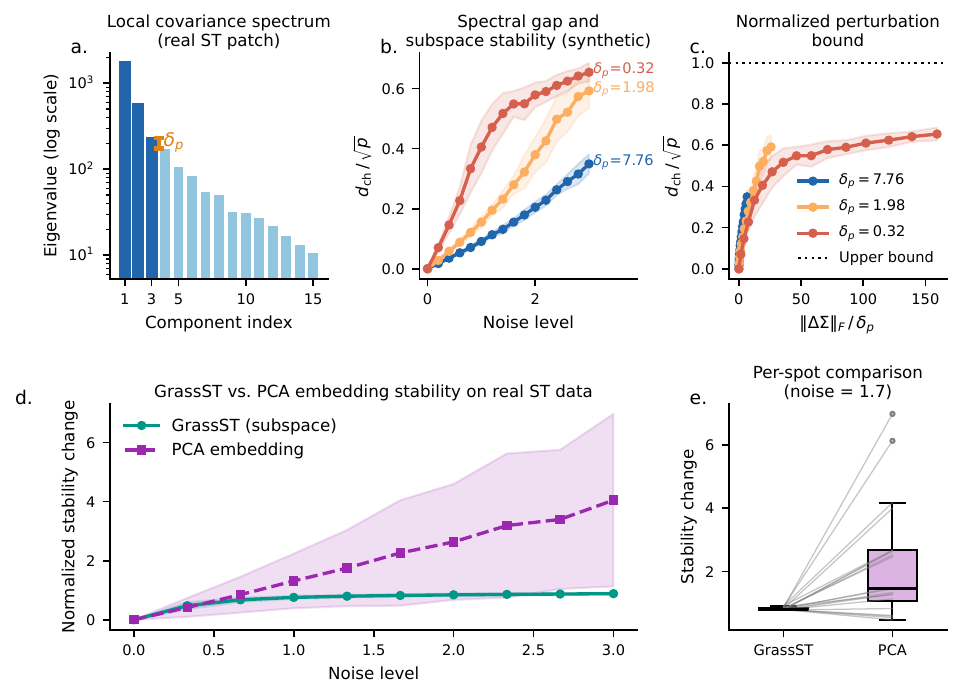}
}
\caption{
Grassmann manifold subspace stability under data perturbation.
(a) Eigenvalue spectrum of the centered local covariance matrix
$\Sigma_i^{(k)}$ estimated from a representative neighborhood of 30 spatially
adjacent spots in the 30-dimensional feature space of the DLPFC dataset
(slice 151673). The top $p=3$ eigenvalues (dark blue) define the retained
subspace. The spectral gap
$\delta_p=\lambda_p-\lambda_{p+1}$ (orange bracket) quantifies the separation
between the retained and discarded eigendirections on a log scale.
(b) Effect of spectral gap magnitude on subspace stability. Synthetic Gaussian
data were generated from covariance matrices with prescribed spectral gaps
($\delta_p\in\{0.32,1.98,7.76\}$; $n=500$ samples, $r=30$, $p=3$). Additive
Gaussian noise of increasing magnitude was applied, and the normalized chordal
distance
$d_{\mathrm{ch}}(\mathcal{U}_0,\widetilde{\mathcal{U}})/\sqrt{p}$
was recorded (mean $\pm$ s.d.; 50 replicates).
(c) Empirical validation of the perturbation scaling in
Corollary~\ref{cor:local_stability}. Results from (b) are re-plotted with the
horizontal axis normalized by $\delta_p$. The curves show similar scaling after
normalization and remain bounded by
$d_{\mathrm{ch}}/\sqrt{p}\leq 1$ (dotted line), consistent with the bound
$d_{\mathrm{ch}}\leq C\|\Delta\Sigma\|_F/\delta_p$.
(d) Comparison of subspace stability and center-spot embedding stability
on real spatial transcriptomics data. For 20 randomly selected spots in the
DLPFC dataset, Gaussian noise of increasing magnitude was added to each local
neighborhood. GrassST stability is measured by
$d_{\mathrm{ch}}/\sqrt{p}$, while embedding stability is measured by the
relative change
$\|z_0-\widetilde z\|_2/\|z_0\|_2$
of the center-spot embedding. Shaded regions denote $\pm$ s.d. over spots
with 30 replicates per spot.
(e) Per-spot comparison at noise level $1.7$. Box plots and connected lines
show the two stability metrics across all 20 spots.
}
\label{fig:subspace_stability}
\end{figure}

\subsection{Data-Adaptive Selection of Neighborhood Scale and Subspace Rank}
\label{subsec:adaptive}

The Grassmann manifold representation depends on two structural parameters:
the neighborhood scale $k$ and the subspace rank $p$. The parameter $k$
determines the spatial extent of the local neighborhood used to estimate
$\Sigma_{b,i}^{(k)}$, while $p$ specifies the dimension of the subspace
assigned to each spot. Thus, each spot is represented as a point on
$\Gr(r,p)$.

These parameters control the resolution of the local representation. A small
neighborhood may contain too few spots to summarize local structure reliably,
whereas a large neighborhood may aggregate spots from distinct tissue regions.
Similarly, a small rank may discard relevant directions of local variation,
whereas a large rank may include weak directions that contribute little to the
dominant neighborhood structure. Therefore, using a fixed pair $(k,p)$ across
all datasets may not be appropriate when tissue organization and spatial
resolution vary across datasets.

Rather than fixing $(k,p)$ globally, we select both parameters from the
observed local spectra. The selection procedure uses cumulative spectral
energy as the primary criterion for determining how much local variation is
retained by each candidate pair.

\subsubsection{Local Spectral Energy and Dataset-Level Averaging}

Let $\mathcal{K}$ denote the candidate set of neighborhood scales and
$\mathcal{P}$ the candidate set of subspace ranks. For spot $(b,i)$ and
candidate scale $k\in\mathcal{K}$, let
\[
\lambda_1^{(b,i,k)} \geq \lambda_2^{(b,i,k)} \geq \cdots \geq 0
\]
denote the eigenvalue spectrum obtained from the centered local neighborhood
features. In the implementation, these eigenvalues are computed from the
squared singular values of the centered local patch matrix. The cumulative
spectral energy at rank $p\in\mathcal{P}$ is defined as
\[
E_p^{(b,i,k)}
=
\frac{
\sum_{\ell=1}^{p} \lambda_\ell^{(b,i,k)}
}{
\sum_{\ell} \lambda_\ell^{(b,i,k)}
}.
\]
This quantity measures the fraction of local variation captured by a
$p$-dimensional subspace at scale $k$.

For each tissue section $b$, we first average over spots,
\[
\bar E_p^{(b)}(k)
=
\frac{1}{n_b}
\sum_{i=1}^{n_b}
E_p^{(b,i,k)} .
\]
The dataset-level summary is then computed by averaging across sections:
\[
\bar E_p(k)
=
\frac{1}{B}
\sum_{b=1}^{B}
\bar E_p^{(b)}(k).
\]
This section-level averaging gives each tissue section equal weight in the
parameter selection procedure.

\subsubsection{Admissible Scales and Rank Selection}

Let
\[
k_{\min}=\min\mathcal{K},
\qquad
k_{\max}=\max\mathcal{K},
\qquad
p_{\min}=\min\mathcal{P},
\qquad
p_{\max}=\max\mathcal{P}.
\]
The selection procedure uses three thresholds and a reference rank
$p_0\in\mathcal{P}$. The threshold $\tau_{\mathrm{patch}}$ is used to identify
admissible neighborhood scales through
$\bar E_{p_0}(k)\geq \tau_{\mathrm{patch}}$. Given the selected scale $k^*$,
$\tau_{\mathrm{rank}}$ selects the smallest candidate rank satisfying
$\bar E_p(k^*)\geq \tau_{\mathrm{rank}}$. The threshold
$\tau_{\mathrm{low}}$ detects the low-rank concentration regime, where the
minimum candidate rank at the finest candidate scale already captures
sufficient local spectral energy. Numerical values are provided in Section S1
of the Supporting Information.

The admissible set of patch sizes is defined by
\[
\mathcal{K}_{\mathrm{adm}}
=
\left\{
k\in\mathcal{K}:
\bar E_{p_0}(k)\geq \tau_{\mathrm{patch}}
\right\}.
\]
If no candidate scale satisfies this condition, all candidate patch sizes are
retained for the subsequent scoring step.

For each retained candidate scale, we compute
\[
S_{\mathrm{patch}}(k)
=
\bar E_{p_0}(k)
+
\alpha
\left(
1-
\frac{|k-k_{\mathrm{mid}}|}{k_{\max}-k_{\min}}
\right),
\]
where
\[
k_{\mathrm{mid}}=\frac{k_{\min}+k_{\max}}{2},
\qquad
\alpha\geq 0.
\]
The second term is used only to break ties or near-ties by mildly favoring
scales away from the extremes of the candidate range. The selected patch size
is
\[
k^*
=
\arg\max_{k\in\mathcal{K}_{\mathrm{adm}}}
S_{\mathrm{patch}}(k),
\]
with the convention that \(\mathcal{K}_{\mathrm{adm}}\) is replaced by
\(\mathcal{K}\) if the admissible set is empty.

Given \(k^*\), the selected subspace rank is the smallest candidate rank whose
dataset-level cumulative spectral energy reaches the rank threshold:
\[
p^*
=
\min
\left\{
p\in\mathcal{P}:
\bar E_p(k^*)\geq \tau_{\mathrm{rank}}
\right\}.
\]
If no candidate rank satisfies this condition, the largest candidate rank
\(p_{\max}\) is selected.

\subsubsection{Concentrated Low-Rank Regime}
\label{subsubsec:lowrank}

Some datasets have local spectra that are already strongly concentrated at the
finest candidate scale and the minimum candidate rank. In this case, the
standard score-based selection is bypassed by a low-rank override.

\begin{definition}[Concentrated low-rank regime]
\label{def:low_rank}
A dataset is said to be in the concentrated low-rank regime if
\[
\bar E_{p_{\min}}(k_{\min}) \geq \tau_{\mathrm{low}}
\]
and
\[
\bar E_{p_{\min}}(k_{\min})
=
\max_{k\in\mathcal{K}}
\bar E_{p_{\min}}(k).
\]
\end{definition}

The first condition requires that the minimum-rank subspace at the finest
candidate scale already captures at least \(\tau_{\mathrm{low}}\) of the
average local spectral energy. The second condition requires that this
low-rank spectral concentration is not improved by increasing the neighborhood
scale. Together, these conditions identify datasets for which the finest
candidate scale and minimum candidate rank already provide a sufficient
low-dimensional local summary under the energy criterion.

In this regime, the selected parameters are set directly to
\[
(k^*,p^*)=(k_{\min},p_{\min}).
\]
Otherwise, the procedure follows the admissible-scale and rank-threshold rules
described above. The complete selection procedure is summarized in
Algorithm~\ref{alg:adaptive}.

\begin{remark}
The reference rank $p_0$ and the balancing parameter $\alpha$ are used only in
the patch-size selection step. The reference rank $p_0$ provides a fixed
low-dimensional summary of the local spectra across candidate patch sizes
through $\bar E_{p_0}(k)$; it is not the final selected subspace rank. After
the patch size $k^*$ is determined, the final rank $p^*$ is selected
separately using the rank threshold $\tau_{\mathrm{rank}}$. The parameter
$\alpha$ controls a mild preference for non-extreme patch sizes among
candidate scales with similar cumulative spectral energy. In all experiments,
we use the fixed values $p_0=3$ and $\alpha=0.3$ for every dataset.
\end{remark}

\begin{algorithm}[hbpt!]
\caption{Adaptive selection of neighborhood scale and subspace rank}
\label{alg:adaptive}
\hrule\vspace{0.3em}
\begin{algorithmic}[1]
\Require Candidate scales $\mathcal{K}$, candidate ranks $\mathcal{P}$, reference rank $p_0$, thresholds $\tau_{\mathrm{patch}}, \tau_{\mathrm{rank}}, \tau_{\mathrm{low}}$, balancing parameter $\alpha$
\State Set $k_{\min}=\min\mathcal{K}$, $k_{\max}=\max\mathcal{K}$, $p_{\min}=\min\mathcal{P}$, and $p_{\max}=\max\mathcal{P}$.
\For{each $k \in \mathcal{K}$}
    \State Construct $k$-nearest spatial neighborhoods.
    \State Compute $\bar{E}_p(k)$ for all $p \in \mathcal{P}$.
\EndFor
\If{$\bar{E}_{p_{\min}}(k_{\min}) \geq \tau_{\mathrm{low}}$ \textbf{and} $\bar{E}_{p_{\min}}(k_{\min}) = \max_{k\in\mathcal{K}} \bar{E}_{p_{\min}}(k)$}
    \State Set $k^* = k_{\min}$ and $p^* = p_{\min}$.
\Else
    \State Define $\mathcal{K}_{\mathrm{adm}}=\{k\in\mathcal{K}: \bar{E}_{p_0}(k)\geq \tau_{\mathrm{patch}}\}$.
    \If{$\mathcal{K}_{\mathrm{adm}}=\emptyset$}
        \State Set $\mathcal{K}_{\mathrm{adm}}=\mathcal{K}$.
    \EndIf
    \State Compute $S_{\mathrm{patch}}(k)$ for all $k\in\mathcal{K}_{\mathrm{adm}}$.
    \State Set $k^* \gets \arg\max_{k \in \mathcal{K}_{\mathrm{adm}}} S_{\mathrm{patch}}(k)$.
    \If{there exists $p\in\mathcal{P}$ such that $\bar{E}_p(k^*) \geq \tau_{\mathrm{rank}}$}
        \State Set $p^* \gets \min\{p\in\mathcal{P}: \bar{E}_p(k^*) \geq \tau_{\mathrm{rank}}\}$.
    \Else
        \State Set $p^* \gets p_{\max}$.
    \EndIf
\EndIf
\State \Return $k^*$, $p^*$
\end{algorithmic}
\vspace{0.3em}\hrule
\end{algorithm}

\subsection{Algorithmic Summary and Complexity }
\label{subsec:algorithm_complexity}

Algorithm~\ref{alg:grassst_repr} summarizes the GrassST representation
procedure after the neighborhood scale $k$ and subspace rank $p$ have been
selected.

\begin{algorithm}[htbp]
\caption{GrassST framework}
\label{alg:grassst_repr}
\hrule\vspace{0.3em}
\begin{algorithmic}[1]
    \Require Spatial transcriptomics slices
    $\{(X^{(b)},S^{(b)})\}_{b=1}^{B}$; feature dimension $r$;
    subspace rank $p$; patch size $k$
    \Ensure Subspaces
    $\{\mathcal{U}_{b,i}:1\leq b\leq B,\;1\leq i\leq n_b\}\subset\Gr(r,p)$;
    optional Grassmann manifold distance output $D$

    \State Concatenate expression matrices and compute a shared PCA map
    $T:\mathbb{R}^d\to\mathbb{R}^r$.
    \State Compute embeddings
    $z_i^{(b)}=T(x_i^{(b)})$ for all spots $(b,i)$.
    \State Construct spatial neighborhoods
    $\mathcal{N}_k(b,i)$ from the coordinate matrices $\{S^{(b)}\}_{b=1}^{B}$.

    \For{each spot $(b,i)$}
        \State Center the neighborhood embeddings
        $\{z_{b',j}:(b',j)\in\mathcal{N}_k(b,i)\}$ to obtain
        $\{y_{b',j}^{(b,i,k)}\}$.
        \State Compute
        \[
        U_{b,i} \gets
        \arg\min_{\substack{U\in\mathbb{R}^{r\times p}\\ U^\top U=I_p}}
        \sum_{(b',j)\in\mathcal{N}_k(b,i)}
        \left\|
        y_{b',j}^{(b,i,k)}-UU^\top y_{b',j}^{(b,i,k)}
        \right\|_2^2 .
        \]
        \State Set $\mathcal{U}_{b,i}=\operatorname{span}(U_{b,i})$.
    \EndFor

    \If{a distance-based downstream task is used}
        \State Compute Grassmann manifold distances
        $d(\mathcal{U}_{b,i},\mathcal{U}_{c,j})$ as needed.
        \State Build a distance-based graph or distance matrix $D$.
    \EndIf

    \State Use the subspaces or the distance-based graph for downstream analysis.
    \State \Return $\{\mathcal{U}_{b,i}\}$, with optional distance output $D$.
\end{algorithmic}
\vspace{0.3em}\hrule
\end{algorithm}

We next discuss the computational cost of constructing the GrassST
representation. Let $N=\sum_{b=1}^{B}n_b$ be the total number of spatial spots,
$d$ the original feature dimension, $r$ the shared feature dimension, $p$ the
subspace rank, and $k$ the patch size. When truncated PCA is used to compute
the shared feature embedding, the approximate cost is $O(Ndr)$, using standard
matrix computation estimates \cite{GolubVanLoan2013}.

Spatial neighborhoods are constructed from two-dimensional coordinates. With a
standard nearest-neighbor data structure, this step has approximate cost
$O(N\log N)$, although the exact cost depends on the implementation and on the
distribution of spatial coordinates. For each indexed spot $(b,i)$, the
subspace is computed from the centered patch matrix
$Z_{\mathcal{N}_k(b,i)}\in\mathbb{R}^{k\times r}$. If the full singular value
decomposition of this local patch matrix is computed, the cost is
$O(\min\{kr^2,rk^2\})$. Since only the leading $p$ singular vectors are needed
to form the $p$-dimensional subspace
$\mathcal{U}_{b,i}=\operatorname{span}(U_{b,i})$, the cost for one local patch
is approximately $O(krp)$ \cite{GolubVanLoan2013}. Thus, the total cost for
estimating all subspaces is approximately $O(Nkrp)$. These $N$ subspace
estimation problems are independent and can be computed in parallel across
spatial spots.

Computing a full pairwise chordal distance matrix requires comparing every pair
of subspaces. For a pair $\mathcal{U}_{b,i}$ and $\mathcal{U}_{c,j}$, the
principal angles are obtained from the singular values of
$U_{b,i}^{\top}U_{c,j}\in\mathbb{R}^{p\times p}$. Forming this matrix has cost
$O(rp^2)$, and computing its singular values has cost $O(p^3)$. Therefore, the
full pairwise chordal distance matrix has cost $O(N^2(rp^2+p^3))$. This
quadratic cost arises only when all pairwise subspace distances are needed; the
GrassST representation itself does not require forming the full $N\times N$
distance matrix. For graph-based downstream analyses, it is often sufficient to
compute a sparse chordal-distance graph by retaining selected nearest
neighbors.

Overall, when $d$, $r$, $p$, and $k$ are treated as fixed, the shared embedding
and subspace estimation steps scale approximately linearly with the number of
spatial spots, while the full pairwise distance computation scales
quadratically in $N$ only if all pairwise Grassmann manifold distances are
required.

\section{Experiments}

We evaluate GrassST on multiple spatial transcriptomics datasets with diverse
tissue structures, sequencing platforms, and spatial resolutions. The
experiments assess whether Grassmann manifold representations provide
stable geometric structure for spatial transcriptomics alignment, while also
measuring clustering quality and cross-slice integration performance. GrassST
is compared with representative baseline methods under a unified preprocessing
and downstream evaluation pipeline.

\subsection{Experimental Setup}

All datasets were processed using a unified workflow. For each dataset, tissue
sections were loaded, filtered, aligned to a common gene set, and annotated
with section labels. Available tissue or domain annotations were used as
ground-truth labels for clustering evaluation. Detailed dataset descriptions,
preprocessing choices, GrassST parameter settings, baseline configurations,
clustering protocols, and evaluation metrics are provided in Section S1 of the
Supporting Information.

\subsection{Performance Comparison}

We first evaluate whether the proposed GrassST framework improves spatial domain identification and cross-slice integration across datasets with different tissue structures and spatial resolutions. The main quantitative comparison is summarized in Fig.~\ref{fig:main_metrics}, which reports clustering accuracy, integration quality, and per-section
performance variability across all datasets and methods.
Overall, GrassST shows strong clustering and integration performance on
Barista and DLPFC, and remains competitive on HER2 and MERFISH. These results
indicate that local Grassmann manifold representations can support spatial
domain identification and cross-slice comparison without using external
spatial alignment procedures.

\begin{figure}[hpb!]
\centering
\includegraphics[width=1\textwidth]{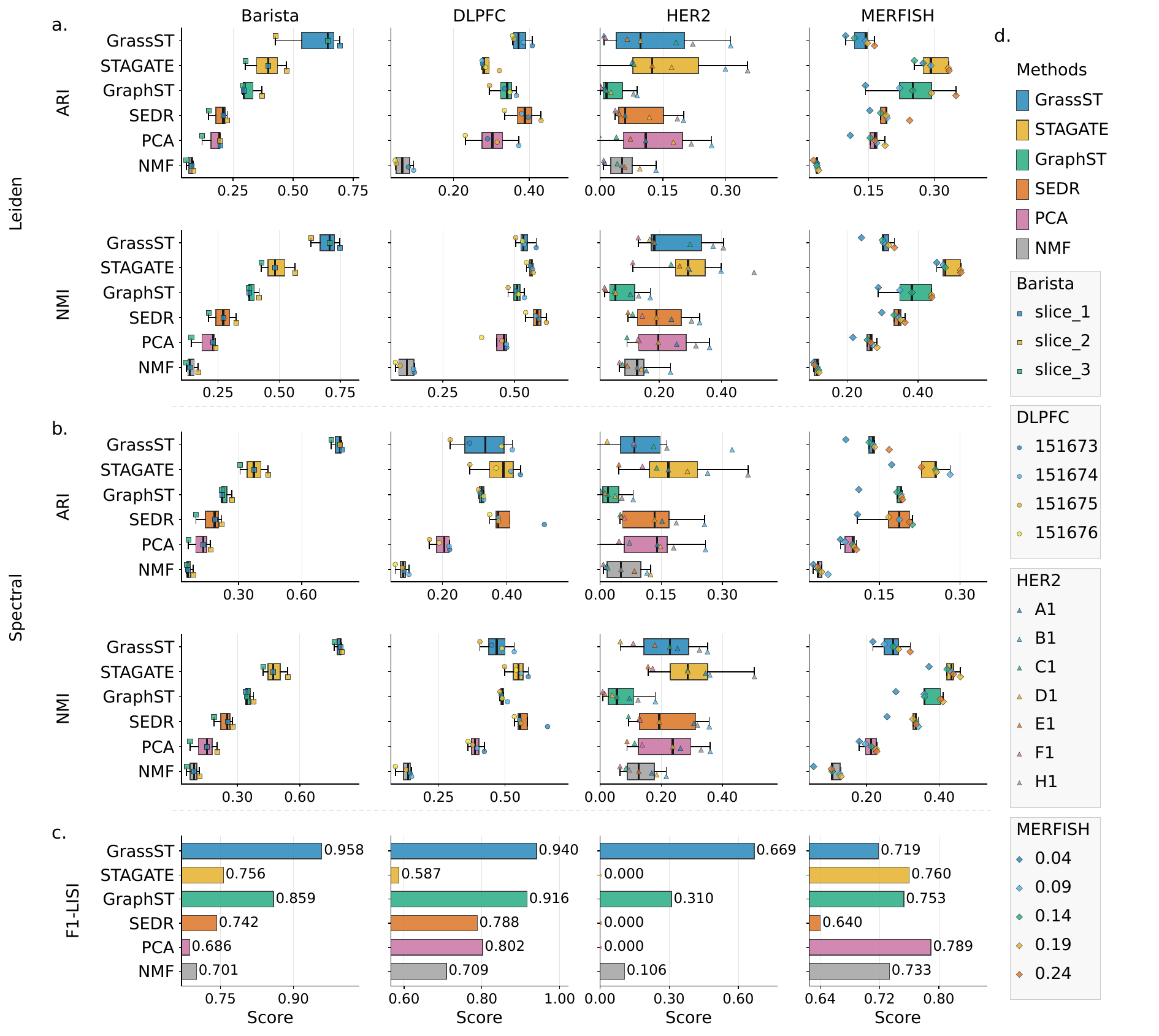}
\caption{
Performance comparison of spatial transcriptomics integration methods using the proposed adaptive GrassST framework.
a. ARI and NMI obtained from Leiden clustering. GrassST consistently achieves strong clustering performance across datasets, outperforming or matching state-of-the-art methods on Baristaand DLPFC, while remaining competitive on HER2 and MERFISH. The results indicate that the proposed local Grassmann manifold representation preserves discriminative structure within each slice.
b. ARI and NMI obtained from spectral clustering based on the Grassmann Manifold chordal distance graph. Compared to a., the relative performance of GrassST remains stable, demonstrating that the learned Grassmann manifold representations are robust to the choice of downstream clustering method. This consistency suggests that the improvement comes from the representation itself rather than a specific clustering algorithm.
c. Median F1-LISI score measuring cross-slice integration quality. GrassST shows substantial improvement over competing methods, particularly on MERFISH, where integration is more challenging due to high spatial resolution and heterogeneity. This demonstrates that the adaptive selection of patch size and subspace dimension effectively captures appropriate local geometry and intrinsic dimensionality across datasets.
d. Legends for methods, datasets, and slice-level identifiers. Different markers correspond to individual slices or sections within each dataset, highlighting robustness to inter-slice variability.
Overall, GrassST achieves a favorable balance between clustering accuracy and integration quality. The results validate that adaptive local modeling on Grassmann manifolds improves generalization across datasets with diverse spatial structures without requiring manual parameter tuning.
}
\label{fig:main_metrics}
\end{figure}

To further assess whether improved quantitative performance corresponds to meaningful embedding structures, we visualize batch alignment and domain preservation using UMAP embeddings in Fig.~\ref{fig:umap}. Across the four datasets, GrassST produces embeddings with clear domain organization while maintaining substantial mixing of slice identities. This behavior indicates that the proposed representation does not simply remove slice-specific variation, but preserves biologically relevant spatial domain structure during integration.

\begin{figure}[hpbt!]
\centering
\includegraphics[width=1\textwidth]{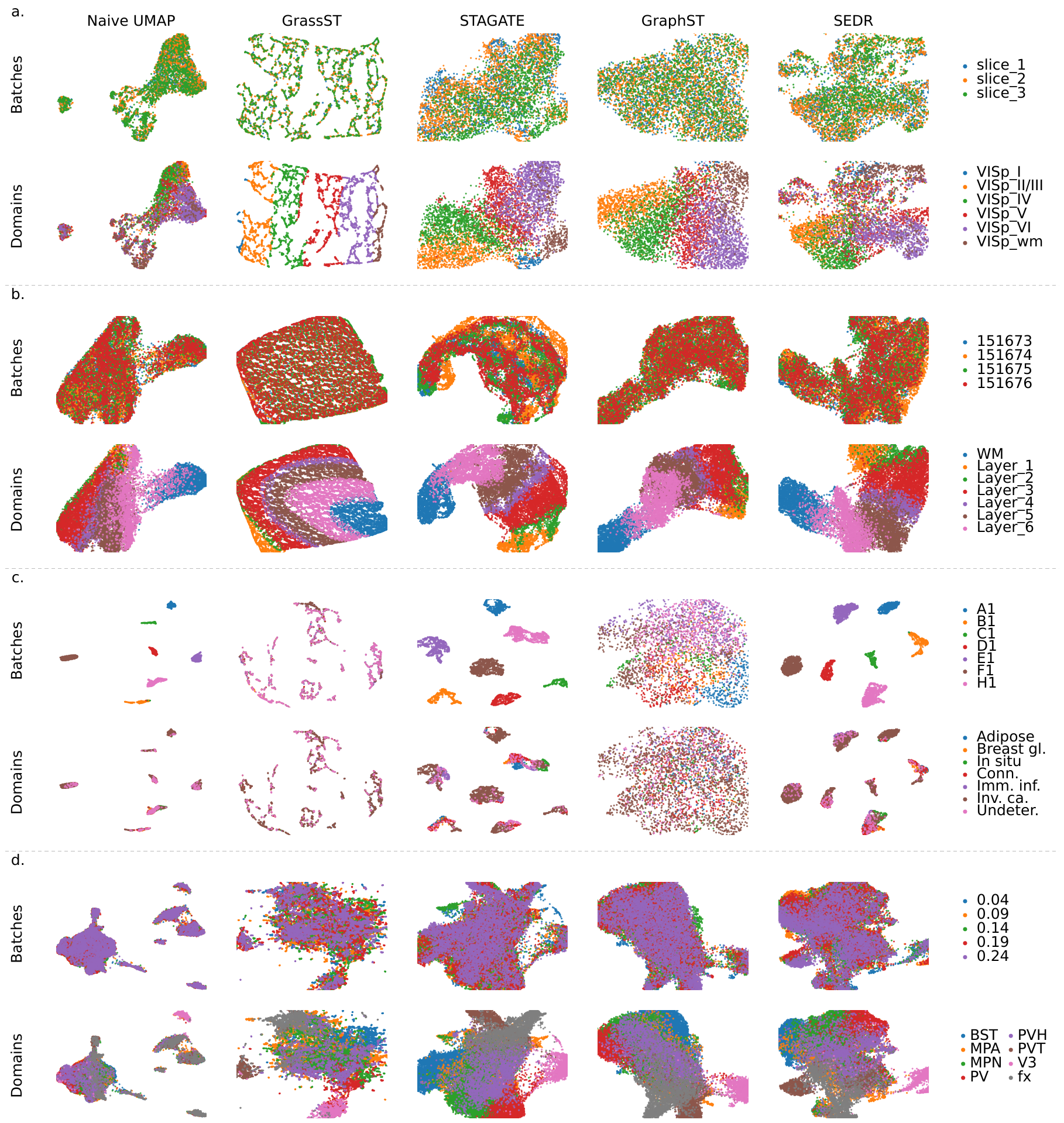}
\caption{UMAP visualization of batch alignment and domain preservation across four spatial transcriptomics datasets.
Each panel (a–d) corresponds to one dataset: (a) Barista (3 slices), (b) DLPFC (4 sections: 151673–151676), (c) HER2, and (d) MERFISH. Within each panel, the top row colors spots by batch (slice/section identity) and the bottom row colors spots by annotated spatial domain. Columns compare five representations: Naive UMAP (no alignment), GrassST, STAGATE, GraphST, and SEDR. A well-aligned embedding should show (i) mixed batch colors in the top row, indicating successful batch integration, and (ii) clearly separated domain clusters in the bottom row, indicating preserved biological structure. GrassST achieves compact, well-separated domain clusters with strong batch mixing across all four datasets, outperforming competing methods particularly in datasets with complex domain architectures (DLPFC) and sparse spot distributions (HER2).
}
\label{fig:umap}
\end{figure}

Additional spatial coherence analyses and tissue-coordinate visualizations are
provided in Section S5 of the Supporting Information. These results include
CHAOS and PAS evaluations across datasets and clustering protocols, together
with spatial domain plots comparing predicted assignments with available
ground-truth annotations.

The results also suggest that GrassST performance is not tied to a single
downstream clustering algorithm. Its relative performance is broadly similar
under both Leiden and spectral clustering, indicating that the local Grassmann
manifold representation provides a useful input for different clustering
protocols. The adaptive selection of patch size and subspace dimension further
allows the method to use dataset-specific local scales, which is important for
spatial transcriptomics datasets with different spatial resolutions and tissue
structures.

\subsection{Adaptive Parameter Selection Analysis}
\label{sec:ablation}

A distinctive aspect of GrassST is that the patch size $k$ and subspace rank
$p$ are selected automatically from the spectral energy landscape
$\bar{E}_p(k)$, rather than fixed by hand. To assess whether this selection is
empirically supported, we evaluated GrassST across all 24 combinations of
$k \in \{15,20,25,30\}$ and $p \in \{2,3,4,5,6,7\}$ on all four datasets and
recorded the chordal F1-LISI at each operating point.

Figure~\ref{fig:param_sweep} summarizes the results. On Barista, DLPFC, and
HER2, F1-LISI increases with $p$ and saturates at higher ranks, forming a
plateau across the upper portion of the parameter grid. The adaptively
selected point $(k^*=20,p^*=7)$ lies within this plateau on all three datasets,
achieving F1-LISI within 0.002 of the grid maximum in each case. The
convergence of the four $k$-value curves at larger $p$ also suggests that
performance is not highly sensitive to the exact patch size once a sufficient
neighborhood scale has been reached.

The MERFISH panel shows a different pattern. In this dataset, F1-LISI
decreases as $p$ increases, and the maximum occurs at
$(k^*=15,p^*=2)$. This behavior is consistent with the low-rank branch of the
selection rule, in which the minimum candidate rank at the finest candidate
scale already captures sufficient local spectral energy. Applying the
normal-branch parameter choice $(k=20,p=7)$ to MERFISH leads to much lower
F1-LISI, showing that the low-rank override is important for this dataset.

Taken together, the parameter sweep indicates that the selected operating
points are close to, or coincide with, the empirical optima of the F1-LISI
landscape in the evaluated grid. Fig.~S3 in Section S5 of the Supporting Information provides a detailed view
of the spectral energy landscapes underlying each selection decision. Fig.~S4
in Section S5 further shows that the selections are stable across a range of
$\tau_{\rm rank}$ values, that the energy criterion is satisfied by each
tissue section at the chosen parameters, and that the per-spot energy
distributions support the use of the section-level mean as the summary
statistic for selection.

\begin{figure}[hptb!]
\centering
 \makebox[\textwidth][c]{
\includegraphics[width=1.\textwidth]{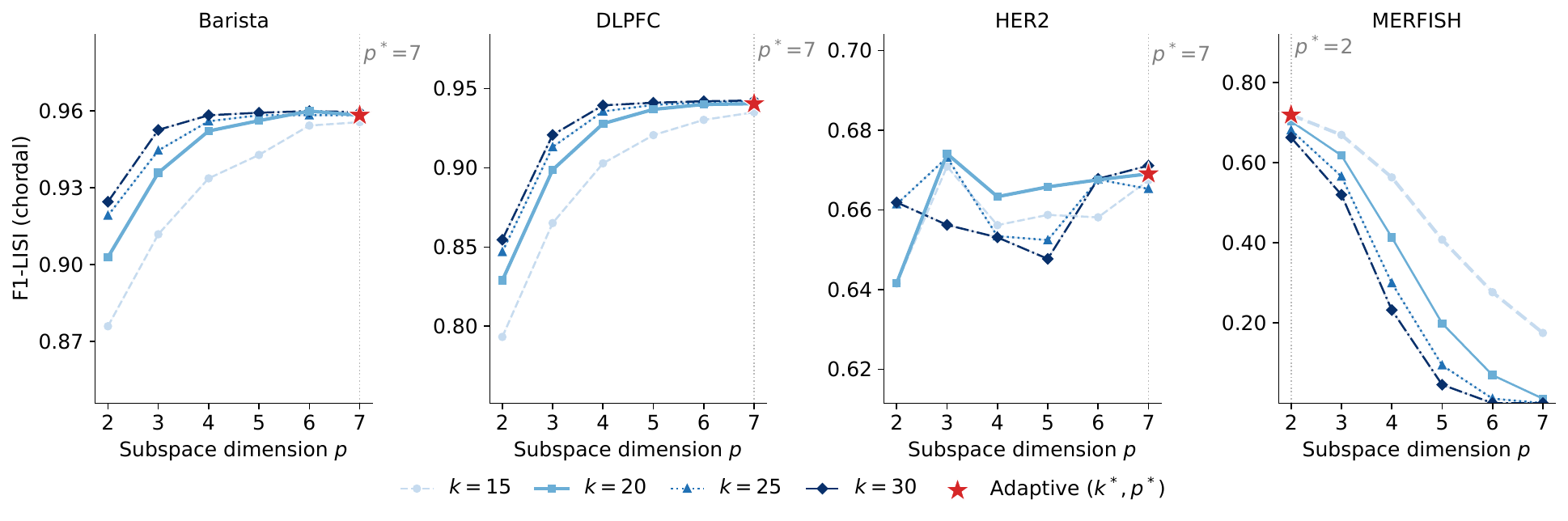}
}
\caption{F1-LISI sensitivity across the full parameter grid.
Each panel shows the chordal F1-LISI as a function of subspace dimension $p$
for all four candidate patch sizes $k \in \{15, 20, 25, 30\}$.
Shaded bands indicate $\pm$1 standard deviation across tissue sections.
The vertical dotted line marks the adaptively selected rank $p^*$, and the
red star ($\bigstar$) indicates the operating point $(k^*, p^*)$ chosen by
the spectral energy criterion.
For the three normal-branch datasets (Barista, DLPFC, HER2), F1-LISI increases
with $p$ and saturates at $p \geq 6$, so $p^*=7$ falls at the onset of
diminishing returns rather than in the steeply rising region.
The curves for different $k$ values converge at large $p$, confirming that
performance is robust to patch size once $k \geq k^*=20$, and the selected
operating point achieves F1-LISI within 0.002 of the grid maximum on every dataset.
For MERFISH (sensitive branch), the trend is reversed: F1-LISI decreases sharply
with $p$, and the global maximum occurs at $p^*=2$, $k^*=15$.
This opposite behaviour is identified by the spectral energy criterion before
parameter selection, triggering the sensitive override and placing the operating
point at the intrinsic low-rank configuration of the data.}
\label{fig:param_sweep}
 \end{figure}

\subsection{Summary}

Fig.~\ref{fig:main_metrics} summarises the primary results across all four datasets.
Per-slice ARI and NMI for both Leiden and spectral clustering, together with the full
per-slice breakdown across all methods, are provided in Tables~S4--S12 in Section S2 of the
Supporting Information.

In Summary, GrassST achieves the highest F1-LISI on three of the four benchmarks
and attains the best average F1-LISI rank among all compared methods.
The advantage is most pronounced on HER2, where STAGATE, SEDR, and PCA fail to
produce a coherent joint embedding (F1-LISI\,$\approx$\,0), while GrassST reaches
$0.669$, more than twice the score of the next valid competitor (GraphST, $0.310$).
On Barista, GrassST leads across all five reported metrics, with a spectral ARI of
$0.770$ against $0.373$ for the next-best method.
On DLPFC, GrassST achieves the highest F1-LISI ($0.940$) despite SEDR obtaining
higher ARI and NMI under both clustering algorithms, reflecting that the two metric
classes measure different aspects of multi-slice analysis.
The one dataset where GrassST does not lead on F1-LISI is MERFISH, where the adaptive selection rule selects the low-rank branch
($k^{*}=15$, $p^{*}=2$), a case discussed further in Section~\ref{dis}.

Spatial coherence metrics provide a complementary view of clustering quality.
Fig.~S1 in Section S3 of the Supporting Information  reports CHAOS and PAS scores across all
datasets and clustering algorithms. GrassST generally achieves low values on
both metrics, indicating that its predicted domains tend to form spatially
coherent regions rather than scattered assignments. This spatial regularity is
observed alongside the integration results, suggesting that the local
Grassmann manifold embedding preserves neighborhood-level tissue structure
while supporting cross-slice comparison.

Performance in datasets with different types of tissues, slice counts, and
sequencing technologies is also supported by the adaptive parameter selection
procedure. Instead of using a fixed patch size and subspace rank for all
datasets, the procedure selects $k$ and $p$ from the spectral energy landscape
of each dataset without dataset-specific manual tuning.

\section{Discussion}
\label{dis}

The results across four spatially and technologically diverse datasets
establish GrassST as the leading method for multi-slice spatial
transcriptomics alignment on the F1-LISI criterion, ranking first on
three of four benchmarks.
On DLPFC, GrassST achieves F1-LISI scores of 0.958 and 0.940
respectively, leading all compared methods by a substantial margin, and the
UMAP visualisations confirm that the learned representations align tissue
sections while preserving the separation of spatial domains
(Fig.~\ref{fig:main_metrics}, Fig.~\ref{fig:umap}).
These datasets represent well-studied tissue types with clear laminar
organisation, and the results suggest that the local Grassmann Manifold geometry
captures the shared structural patterns across sections reliably.

The HER2 result merits particular attention.
HER2 is a tumour biopsy dataset with seven heterogeneous slices, and it is
the setting where the gap between GrassST and deep learning baselines is
most pronounced.
STAGATE achieves an F1-LISI of 0.000 on this dataset, indicating a complete
failure to align the sections, and GraphST reaches only 0.310, while GrassST
obtains 0.669.
This difference reflects a fundamental characteristic of neural network-based
methods: their representations are learned by optimizing an objective over
the training data, and when the number of slices is small and the tissue is
heterogeneous, there is insufficient signal to guide that optimization toward
a meaningful solution.
GrassST constructs subspaces directly from the spatial neighborhood
of each spot without any training, so its performance does not depend on
the number of sections available.
This property is important in practice, as collecting many spatially
registered slices from the same specimen is often not feasible.

The DLPFC results highlight a distinction between per-slice clustering
quality and cross-slice integration quality.
SEDR achieves higher ARI and NMI than GrassST under both clustering
algorithms on this dataset, yet its F1-LISI ($0.788$) falls well below
GrassST ($0.940$).
A method can assign spots to coherent clusters within each section without
aligning the latent representations of corresponding cells across sections.
F1-LISI captures this cross-slice alignment directly, and the DLPFC result
shows that the two objectives are not equivalent.
This distinction motivates the choice of F1-LISI as the primary evaluation
metric in this study.

The consistency of GrassST across datasets is partly attributable to the
adaptive parameter selection procedure.
Different tissue types have different spectral structures: well-organized
laminar tissue such as DLPFC concentrates variance in a moderate number of
dimensions, while MERFISH data with a curated gene panel is intrinsically
lower-rank.
Assuming a single fixed configuration across all tissue types would impose
a geometry that is inappropriate for most of them.
The spectral energy criterion selects $k$ and $p$ from each dataset's own
variance landscape, and the ablation analysis in Section~\ref{sec:ablation}
confirms that the selected operating points correspond to regions of high
F1-LISI across the parameter grid.

% One limitation of the current framework concerns MERFISH.
% Although the sensitive branch correctly detects the low-rank spatial
% organisation and assigns $(k^*=15, p^*=2)$, GrassST achieves an F1-LISI
% of 0.719, below PCA at 0.789.
% MERFISH measures a curated panel of a few hundred genes at high spatial
% resolution, and the resulting transcriptional geometry is more constrained
% than that of sequencing-based datasets.
% In this regime the two-dimensional subspace representation, while correctly
% identified as appropriate, may not capture sufficient structure to support
% strong alignment.
% Developing richer representations for intrinsically low-rank data, perhaps
% by incorporating spatial coordinates more directly into the subspace
% construction, is a natural direction for future work.

One limitation of the current framework concerns the construction of spatial
neighborhoods in high-resolution datasets such as MERFISH. In the current
implementation, kNN neighborhoods are built from stacked spatial coordinates,
so the within-slice and cross-slice neighbor composition depends on the
relative density and spatial closeness of the slices. This may be less reliable
when slices have different spatial sampling densities or when their coordinates
are not well aligned. This issue may partly contribute to the more challenging behavior observed on
MERFISH compared with the other datasets.

At the same time, the experiments show that GrassST remains competitive in
several datasets without using external spatial coordinate alignment methods
such as PASTE\cite{Zeira2022NatMethods}. This suggests that the local Grassmann manifold representation
can be useful for clustering and cross-slice comparison under a simple
neighborhood construction. A possible future direction is to combine GrassST
with spatial coordinate alignment methods to define cross-slice neighborhoods
more carefully. For example, within-slice neighbors could be assigned using
the original spatial coordinates, while cross-slice neighbors could be assigned
using PASTE-aligned coordinates. Fixed distance thresholds may also help reduce
spurious neighbors near tissue boundaries.

\section{Conclusion}

We have presented GrassST, a framework for multi-slice spatial transcriptomics
alignment that represents the local gene expression neighborhood of each spot
as a low-dimensional subspace on the Grassmann manifold and selects its two
key parameters, the neighborhood size $k$ and the subspace rank $p$,
automatically from the spectral energy landscape of each dataset.
This combination of geometry-aware representation and data-driven parameter
selection removes the need for manual configuration and allows the framework
to generalize across datasets with substantially different structures.

The data-adaptive selection procedure identifies the operating point through a
spectral energy criterion that distinguishes between a normal regime, in
which variance is distributed across multiple dimensions, and a sensitive
regime, in which the data is intrinsically low-rank.
This distinction is not merely a tuning detail: it reflects a genuine
difference in the geometry of the input data, and selecting the wrong branch
leads to a measurable drop in alignment quality.
To our knowledge, GrassST is the first spatial transcriptomics alignment
framework to incorporate such a regime-aware parameter selection without
any labelled input or user intervention.

The experimental results support the practical value of this design.
Across four datasets with different characteristics, GrassST achieves the
highest F1-LISI on three benchmarks and is the only method to produce a
coherent joint embedding on HER2, where seven heterogeneous sections expose
the limits of methods that depend on sufficient training data to converge.
Spatial coherence metrics further confirm that the resulting domain assignments
are anatomically consistent across all evaluated cases.

On datasets such as MERFISH, where only a small panel of genes is measured,
the sensitive branch of the adaptive selection correctly identifies the
low-rank structure and adjusts the parameters accordingly.
GrassST remains competitive in this setting, and the modest gap to the
best-performing baseline points to an opportunity for further improvement
by incorporating spatial coordinates more directly into the subspace
construction.

More broadly, the results suggest that Grassmann manifold geometry, together
with data-adaptive choices of neighborhood scale and subspace rank, can be a
useful component of multi-slice spatial transcriptomics analysis, particularly
when training data are limited or when datasets differ in spatial resolution
and tissue organization.

%  \section*{Data and Code Availability}

% The source code used in this study is available at \url{https://github.com/XiangXiangJY/GrassST0}. The spatial transcriptomics datasets used in the numerical experiments are available  at \url{https://weilab.math.msu.edu/DataLibrary/SpatialTranscriptomics/}.

\section*{Supporting Information}
Additional experimental settings, baseline details, full quantitative results,
and adaptive parameter selection analyses are provided in the Supporting
Information.
 \section*{Acknowledgments}
 This work was supported in part by NIH grants R01AI164266 and R35GM148196, MSU Research Foundation, The University of Georgia, and the Georgia Research Alliance.  

 % 	    \section*{Acknowledgments}
 % This work was supported in part by NIH grants R01AI164266 and R35GM148196, MSU Research Foundation, The University of Georgia, and the Georgia Research Alliance.    %grant R35GM148196, National Science Foundation grant DMS2052983,  Michigan State University Research Foundation, and  Bristol-Myers Squibb 65109.  

%\bibliographystyle{unsrt}
% \bibliographystyle{unsrt}
%\myexternaldocument{supplementary}
\bibliographystyle{unsrt}

\bibliography{refs}
   
\end{document}